\documentclass[11pt]{amsart}
\usepackage{amssymb,latexsym}
\usepackage{amsmath,color}
\usepackage{amsthm}
\usepackage{epsfig}
\usepackage{graphicx}
\usepackage{hyperref}
\usepackage{titletoc}
\usepackage{psfrag}
\usepackage{palatino,mathpazo}

\usepackage[margin=1.5in]{geometry}

\numberwithin{equation}{section}
\newtheorem{theorem}{Theorem}[section]
\newtheorem{proposition}[theorem]{Proposition}
\newtheorem{lemma}[theorem]{Lemma}

\newtheorem{definition}{Definition}[section]

\theoremstyle{definition}

\newtheorem{remark}{Remark}[section]

\def\Xint#1{\mathchoice
	{\XXint\displaystyle\textstyle{#1}}%
	{\XXint\textstyle\scriptstyle{#1}}%
	{\XXint\scriptstyle\scriptscriptstyle{#1}}%
	{\XXint\scriptscriptstyle\scriptscriptstyle{#1}}%
	\!\int}
\def\XXint#1#2#3{{\setbox0=\hbox{$#1{#2#3}{\int}$}
		\vcenter{\hbox{$#2#3$}}\kern-.5\wd0}}

\def\dashint{\Xint-}

\def\ou{\oveline{u}}

\def\e{\varepsilon}
\def\B{\mathbb{R}}

\def\Rm{\mathbb{R}^{2m}}

\def\Xint#1{\mathchoice
	{\XXint\displaystyle\textstyle{#1}}%
	{\XXint\textstyle\scriptstyle{#1}}%
	{\XXint\scriptstyle\scriptscriptstyle{#1}}%
	{\XXint\scriptscriptstyle\scriptscriptstyle{#1}}%
	\!\int}
\def\XXint#1#2#3{{\setbox0=\hbox{$#1{#2#3}{\int}$ }
		\vcenter{\hbox{$#2#3$ }}\kern-.6\wd0}}

\def\dashint{\Xint-}

\renewcommand{\a }{\alpha}
\renewcommand{\t }{\tau }
\newcommand{\M}{M^{\mbox{\tiny{R}}}}

\begin{document}
\title[$Q$-curvature with conical singularities]{Prescribing $Q$-curvature on even-dimensional manifolds with conical singularities}
	
\author[A. Jevnikar]{ Aleks Jevnikar}
\address{Aleks Jevnikar, Department of Mathematics, Computer Science and Physics, University of Udine, Via delle Scienze 206, 33100 Udine, Italy}
\email{aleks.jevnikar@uniud.it}
	
\author[Y. Sire]{Yannick Sire}	
\address{Yannick Sire, Department of Mathematics, Johns Hopkins University, 3400 N. Charles Street,
Baltimore, MD 21218, US}
\email{ ysire1@jhu.edu}
	
\author[W. Yang]{Wen Yang}
\address{Wen Yang, Department of Mathematics, Faculty of Science and Technology, University of Macau, Macau}
\email{wenyang@um.edu.mo}

\thanks{A. Jevnikar is partially supported by INdAM-GNAMPA project 'Analisi qualitativa di problemi differenziali non lineari' and PRIN Project 20227HX33Z 'Pattern formation in nonlinear phenomena' funded by European Union - Next Generation EU within the PRIN 2022 program.
Y. Sire is partially supported by NSF DMS grant $2154219$, ``Regularity {\em vs} singularity formation in elliptic and parabolic equations". W. Yang is supported by National
Key R\&D Program of China 2022YFA1006800, NSFC, China No.
12171456, NSFC, China No. 12271369, FDCT No. 0070/2024/RIA1,
Multi-Year Research Grant No. MYRG-GRG2024-00082-FST and Startup
Research Grant No. SRG2023-00067-FST}

\begin{abstract}
On a $2m$-dimensional closed manifold we investigate the existence of prescribed $Q$-curvature metrics with conical singularities. We present here a general existence and multiplicity result in the supercritical regime. To this end, we first carry out a blow-up analysis of a $2m$th-order PDE associated to the problem and then apply a variational argument of min-max type. For $m>1$, this seems to be the first existence result for supercritical conic manifolds different from the sphere.
\end{abstract}
\maketitle

\section{Introduction}
In conformal geometry, one of the most fundamental problems is understanding the relationship between conformally covariant operators, their associated conformal invariants and the related PDEs. 

As a first example, let us consider the Laplace-Beltrami operator in two dimensions on a closed surface $(M,g)$ and the  Gaussian curvature. If we want to prescribe the curvature $K$ through a conformal change of metric $g_v=e^{2v}g$, we have the associated PDE
\begin{equation}
	\label{1.curvature-q}
	-\Delta_gv+K_g=Ke^{2v},
\end{equation}
where $\Delta_g$ denotes the Laplace-Beltrami operator with respect to the background metric $g$ and $K_g$, $K=K_{g_v}$ are the Gaussian curvatures of the metric $g$ and ${g_v}$, respectively. Observe that the latter equation yields in particular the conformal invariance of the total Gaussian curvature which is then tight to the topology of the surface via the Gauss-Bonnet formula 
$$
\int_M K_g\,dvol_g=\chi(M).
$$ 
Here, $\chi(M)$ is the Euler characteristic of the surface. 

A classical issue here is the prescribed Gaussian curvature problem or the Uniformization Theorem about the existence of a conformal metric in the conformal class of $g$ with prescribed (possibly constant) curvature. This amounts to solve the PDE in \eqref{1.curvature-q} which has been systematically studied since the works of Berger \cite{berger}, Kazdan-Warner \cite{kw} and Chang-Yang \cite{cy1,cy2}.

\

In higher dimensions, we have the so-called GJMS operators $P_g^{2m}$ and the related $Q$-curvatures $Q_g^{2m}$ which are the higher-order analogues of the Laplace-Beltrami operator and the Gaussian curvature for $2m$-dimensional closed manifolds, see \cite{gjms,GZ}. These are conformally covariant differential operators whose leading term is $(-\Delta_g)^m$.
In particular, when $m=1$ we recover the Laplace-Beltrami operator and the Gaussian curvature. Moreover, for $m=2$, $P_g^{4}$ and $Q_g^4$ are related to the Paneitz operator and the standard $Q$-curvature:
\begin{equation}
\label{1.paneitz}
\begin{aligned}
P_g^4f&=P_g f=\Delta_g^2f+\mbox{div}_g\left(\frac23R_gg-2\mbox{Ric}_g\right)df,\\
Q_g^4&=2Q_g=-\frac1{6}(\Delta_gR_g-R_g^2+3|\mbox{Ric}_g|^2),	
\end{aligned}
\end{equation}
where $\mbox{Ric}_g$ and $R_g$ stand for the Ricci tensor and the scalar curvature of the manifold $(M,g)$. See the original works of Paneitz \cite{pan1,pan2} and Branson \cite{branson} for more details. 

\

The family of GJMS operators and the related $Q$-curvature functions play now an important role in modern differential geometry. As in the lower order case, if we want to prescribe the $Q-$curvature $Q_{g_v}^{2m}=Q$ through a conformal transformation $g_v=e^{2v}g$, then $P_g^{2m}$ and $Q_g^{2m}$ satisfy the following laws
\begin{equation}
	\label{1.higher}
	P_{g_v}^{2m}=e^{-2mv}P_g^{2m},\quad P_g^{2m}v+Q^{2m}_g=Qe^{2mv}.
\end{equation}
The prescribed $Q$-curvature problem is in thus related to the solvability of \eqref{1.higher}. 

\medskip

One can attack this problem variationally by looking at the critical points of the associated energy functional. A lot of work has been done in this direction, in particular for the four-dimensional case and the Paneitz operator \eqref{1.paneitz}. In this setting, assuming 
$$ 
P_g\geq0, \quad \mbox{Ker}\{P_g\}=\{constants\},
$$ 
the problem has been first solved by Chang-Yang \cite{cy3} for 
$$
\int_M Q^{4}_g \,dvol_g=2\int_M Q_g \,dvol_g<16\pi^2=2\int_{\mathbb{S}^4}Q_{g_0}\,dvol_{g_0}.
$$ 
Here, $g_0$ is the standard metric of the sphere. See also the related work of Gursky \cite{gursky}. This is the so-called subcritical case in which the energy functional is coercive and bounded from below by means of the Adams-Trudinger-Moser inequality \cite{adams} and solutions corresponds to global minima using the direct methods of the calculus of variations.  We refer to the discussion in the sequel for the precise definition of the subcritical, critical and supercritical case. The supercritical case $\int_M Q^{4}_g \,dvol_g>16\pi^2$, where the energy functional fails to be bounded from below, has been considered by Djadli-Malchiodi \cite{dm} via a new min-max method based on improved versions of the Adams-Trudinger-Moser inequality \cite{aubin2013some,ChenLi1}, solving the problem provided
$$
	\mbox{Ker}\{P_g\}=\{constants\}, \quad \int_M Q^{4}_g \,dvol_g \notin 16\pi^2\mathbb{N}.
$$
Finally, some existence results for the critical case $\int_M Q^{4}_g \,dvol_g \in 16\pi^2\mathbb{N}$ have been derived by Ndiaye \cite{ndiaye2} by making use of the critical point theory at infinity jointly with a blow-up analysis.

\medskip

As far as the higher-dimensional case $2m>4$ is concerned, the subcritical case has been solved in \cite{brendle} via a geometric
flow, while the Djadli-Malchiodi's argument has been generalized by Ndiaye \cite{ndiaye} to treat the supercritical case.

\

In this paper we are interested in prescribing the $Q$-curvature on a general $2m$-dimensional closed manifold $M$ with conical singularities. Let $g$ be a smooth metric on $M$. We will say that a point $q\in M$ is a conical singularity of order $\alpha\in(-1,+\infty)$ for the new metric $g_v=e^{2v}g$ if  
$$
	g_v(x)=f(x)|x|^{2\alpha}|dx|^2 \quad \mbox{locally around } q,
$$
for some smooth function $f$. The set of conical singularities $q_j$ of orders $\alpha_j$ is encoded in the formal sum
$$
	D=\sum_{j=1}^N \alpha_j q_j,
$$
while $(M,D)$ will denote the related conical manifold. We define
$$
	\kappa_{g}=\int_MQ_{g}^{2m}\,dvol_{g}, \quad \kappa_{g_v}=\int_MQ_{g_v}^{2m}\,dvol_{g_v},
$$
for which the following relation holds
\begin{equation} \label{kappa}
\kappa_{g_v}=\kappa_{g}+\dfrac{\Lambda_m}{2} \sum_{j=1}^N \alpha_j,
\end{equation}
where $\Lambda_m=(2m-1)!|\mathbb{S}^{2m}|$, see for example Theorem~\ref{CGB}. The critical threshold of a singular manifold is essentially related to the singular Adams-Trudinger-Moser inequality stated in Theorem \ref{ATM}. In the spirit of Troyanov \cite{troy} we let
$$
\t(M,D)=\Lambda_m\Bigr(1+\min_j\left\{\a_j,0\right\}\Bigr)
$$
and give the following classification.

\begin{definition} The singular manifold $(M,D)$ is said to be:
$$\begin{array}{ll}
\hbox{subcritical}&\hbox{if} \quad \kappa_{g_v}<\t(M,D)\\
\hbox{critical}&\hbox{if} \quad \kappa_{g_v}=\t(M,D)\\
\hbox{supercritical}&\hbox{if} \quad \kappa_{g_v}>\t(M,D).
\end{array}$$
\end{definition}
See also the recent work of Fang-Ma \cite{FangMa} for a similar discussion. We point out we have a slightly different notation for $\Lambda_m$ with respect to the latter paper.

\

Due to the singular behavior of the conformal factor $v$ around a conical point, prescribing the $Q$-curvature on a manifold with conical singularities at $q_j\in M$ of order $\alpha_j\in(-1,+\infty)$ is related to the solvability of the following singular PDE
\begin{equation}
\label{1.s-mliouville}
P_g^{2m}v+Q_g^{2m}=Q_{g_v}^{2m}e^{2mv}-\frac{\Lambda_m}{2}\sum_{j=1}^N\alpha_j\delta_{q_j},
\end{equation}
where $\delta_{q_j}$ stands for the Dirac measure located at the point $q_j\in M$. One may desingularize the behavior of $v$ around the conical points by considering
\begin{equation*}
u=v+\frac{\Lambda_m}{2}\sum_{j=1}^N\alpha_jG(x,q_j),
\end{equation*}
where $G(x,p)$ is the Green function of $P_g^{2m}$, see for example Lemma \ref{expGreen}. Then $u$ satisfies
\begin{equation}
\label{1.mliouville}
P_g^{2m}u+Q_g^{2m}+\dfrac{\Lambda_m}{2|M|} \sum_{j=1}^N \alpha_j=\widetilde Q e^{2mu},
\end{equation}
where 
\begin{equation} \label{qtilde}
\widetilde Q=Q_{g_v}^{2m}e^{-m\Lambda_m\sum_{j=1}^N\alpha_jG(x,q_j)},
\end{equation}	
which is now singular at the points $q_j$.

\medskip

The singular equation \eqref{1.mliouville} has been studied mainly in the two-dimensional case, that is in relation to the prescribed Gaussian curvature problem. After the initial work of Troyanov \cite{troy}, there have been contributions by many authors, as for example \cite{ChenLi1,ChenLi2,ChenLi3,LuoTian,McOwen}. This problem has received a lot of attention also in recent years, see \cite{bdm,bm,bt,cm,mr}. See also \cite{eremenko,MazzeoZhu,MondPanov1,MondPanov2} for further developments in this direction. 

\medskip

In the higher-dimensional case $m>1$ there are very few results available. The subcritical regime have been just recently solved by Fang-Ma \cite{FangMa}, where the four-dimensional case is considered. The authors point out their method could be applied for higher dimensions too. In any case, the existence here follows by direct methods of the calculus of variations once the singular Adams-Trudinger-Moser inequality in Theorem \ref{ATM} is derived. See also \cite{hlw} for a related result on the sphere via a fixed point argument. For a blow-up analysis in dimension four we refer instead to \cite{AhmedouWuZhang}. Concerning the existence problem in the supercritical case, the only result we are aware of is \cite{hmm} where the authors consider a slightly supercritical problem on the sphere, again with a fixed point argument in the spirit of \cite{hlw}.

\medskip

The goal of this paper is to give a first general existence result for $2m$-dimensional conic manifolds in the supercritical regime. We define a critical set of values $\Gamma$ as follows:
\begin{equation} \label{gamma}
\Gamma=\left\{n\Lambda_m+\Lambda_m\sum_{i\in J}(1+\alpha_i)\mid n\in\mathbb{N}\cup\{0\}\quad\mathrm{and}\quad
J\subset\{1,\dots,N\}\right\}.
\end{equation}
Observe that if $\alpha_j\in\mathbb{N}$ for all $j$, then we simply have $\Gamma=\Lambda_m\mathbb{N}$. Recall now the definition of the total curvature $\kappa_{g}$ given before \eqref{kappa}. Let $\M\subset M$ be a closed $n$-dimensional submanifold, $n\in[1,2m)$, such that the singular points $q_j\notin\M$ for all $j=1,\dots,N$. Then, we have:
\begin{theorem}
\label{th1.2}
Let $(M,D)$ be a supercritical singular $2m$-dimensional closed manifold with $\alpha_j>0$ for $j=1,\dots,N$ and let $Q$ be a smooth positive function on $M$. Suppose that there exists a retraction $R:M\to \M$, with $\M\subset M$ as above. If moreover
$$
\mbox{Ker}\{P_g^{2m}\}=\{constants\}, \quad \kappa_{g}+\dfrac{\Lambda_m}{2} \sum_{j=1}^N \alpha_j\notin\Gamma,
$$ 
then there exists a conformal metric on $(M,D)$ with $Q^{2m}$-curvature equal to $Q$.
\end{theorem}

\smallskip

\begin{remark}\label{rem-ret}
\emph{We point out that a retraction $R:M\to \M$ as above exists for a wide class of manifolds. For example we can consider manifolds of the type $M^{n}\times M^{2m-n}$, where we denote by $M^l$ any $l$-dimensional closed manifold. Indeed, it is easy to see that we can define a retraction $R:M^{n}\times M^{2m-n}\to M^{n}\times\{p\}$ for some $p\in M^{2m-n}$ with the desired properties. Observe that the torus $\mathbb{T}^{2m}$ belongs to this class of manifolds. One could also consider the connected sum $(M^{n}\times M^{2m-n})\# N^{2m}$, modifying the above retraction so that it is constant on $N^{2m}$.}
\end{remark}

\medskip

We can also deduce the following multiplicity result. Here, $\M_k$ are the formal barycenters of $\M$ according to \eqref{bary} and $\widetilde H_q(\M_k)$ denotes its reduced $q$-th homology group.
\begin{theorem} \label{th-mult}
Under the assumptions of Theorem \ref{th1.2}, let $\kappa_{g_v}\in(k\Lambda_m,(k+1)\Lambda_m)$. Then, if  $\mathcal E$ in \eqref{functional} is a Morse functional,
$$
	\#\{ \mbox{solutions of }\eqref{1.mliouville}\} \geq \sum_{q\geq0} \mbox{\emph{dim} } \widetilde H_q(\M_k).
$$
\end{theorem}

\begin{remark}
\emph{Consider for example the class of manifolds $M^{n}\times M^{2m-n}$ in Remark \ref{rem-ret}. We will get an explicit lower bound on the number of solutions as far as we can explicitly estimate the homology groups of $M^{n}_k$. One can find such computations in \cite{de} for general manifolds $M^n$, focusing on the cases $n=2$ and $n=4$. For some simple manifolds we can easily compute the homology groups. For example, if $M^{n}$ is a $2$-dimensional $G$-torus (connected sum of $G$ tori), then we have at least $\frac{(N+G-1)!}{N!(G-1)!}$ solutions, see \cite{bdm}.}
\end{remark}

\

The argument of the proof of the existence result is in the spirit of the celebrated min-max scheme of \cite{dm}, extended to high dimensions by \cite{ndiaye}, jointly with some ideas of \cite{bdm} to treat the singularities. Roughly speaking, the strategy is based on the study of the sublevels of the energy functional, in particular by showing the low sublevels are non-contractible. This is done by using improved versions of the singular Adams-Trudinger-Moser inequality. We will then overcome the complexity due to the singularities by retracting the manifold onto $\M$, not containing the singular points. This leads us to study the low sublevels just by looking at functions concentrating on such submanifold which is enough to gain some non-trivial homology. We refer the interested readers to \cite{bjwy} and \cite{bjmr} for a similar approach applied to surfaces with boundary and Toda systems, respectively.

\medskip

To conclude the min-max argument we would need some compactness property as the Palais-Smale conditions are not available in this setting. We thus use Struwe's monotonicity trick \cite{struwe}, which is by now a standard tool in this class of problems, to deduce the existence of a sequence of solutions $u_k$ satisfying \eqref{1.mliouville}. We will then conclude by showing the following compactness result which actually holds for any $2m$-dimensional manifold and $\alpha_j>-1$.   
\begin{theorem}
\label{th1.1}
Let $u_k$ be a sequence of solutions of \eqref{1.mliouville} with $\widetilde Q>0$ and $\alpha_j>-1$ for $j=1,\dots,N$. If 
$$
\mbox{Ker}\{P_g^{2m}\}=\{constants\}, \quad \kappa_{g_v}\notin\Gamma,
$$  
then there exists a constant $C$ independent of $k$ such that
\begin{equation*}
\|u_k\|_{L^\infty(M)}\leq C.
\end{equation*}
\end{theorem}

The latter result is a consequence of a quantization phenomenon of blowing-up solutions which is derived via Pohozaev-type inequalities in the spirit of \cite{AhmedouWuZhang,bt,LinWei}.

\medskip

The above analysis, together with Morse inequalities, allows us to deduce also the multiplicity result of Theorem \ref{th-mult}.

\begin{remark}
\emph{We conclude the introduction with the following observations.} 

\medskip

\noindent 1. \emph{The existence result is derived for the case $\alpha_j>0$ for all $j$. In principle, the same strategy can be carried out for the case $\alpha_j\in(-1,0)$. However, in this scenario we get a worse Adams-Trudinger-Moser inequality in Theorem \ref{ATM} and this in turn affects the topology of the low sublevels in a non-trivial way, see for instance \cite{cm}. We postpone this study to a future paper.}

\medskip

\noindent 2. \emph{The same analysis should work in the odd-dimensional case with some further technical difficulties, as explained in Section 5 of \cite{ndiaye}. We will not discuss this case in the present paper.}
\end{remark}

\medskip

This paper is organized as follows: in Section 2 we collect some useful preliminary results, Section 3 is devoted to blow-up analyis and the proof of Theorem~\ref{th1.1} and in Section 4 we carry out the min-max method to derive the existence and multiplicity results of Theorems \ref{th1.2} and \ref{th-mult}. A Pohozaev-type identity is provided in the last section.

\medskip

\begin{center}
\textbf{Notations:}
\end{center}
\begin{enumerate}
\item [$B_r^{M}(p)$] \quad the ball centered at $p\in M$ with geodesic radius $r$ on the manifold $M$.
\smallskip
\item [$B_r(p)$] \quad the ball centered at $p$ with radius $r$ in $\mathbb{R}^{2m}$.
\end{enumerate}

\

\section{Preliminary facts}

In this section we recall briefly some known results which can be easily derived from the existing literature.  Let $p$ be a point in $M$ and $B_r^M(p)$ be the geodesic normal ball such that $B_r^M(p)$ is mapped by $\exp_p^{-1}$ diffeomorphically onto a neighborhood of $0\in T_p(M),$ where $T_p(M)$ refers to the tangent space of $p$ which can be identified with $\Rm$. The local coordinates defined by the chart $(\exp_p^{-1},B_r^{M}(p))$ are called normal coordinates with center $p.$ In such coordinates, the Riemannian metric at the point $p$ satisfies
\begin{equation}
\label{a.m1}
g_{ij}=\delta_{ij},~g_{ij,k}=0,~\Gamma_{jk}^i=0,\quad \mbox{for all}~i,j,k\in\{1,\cdots,2m\},
\end{equation}
where $\Gamma_{jk}^i$ stands for the Christoffel symbols. With the above preparation, we have
\begin{lemma}
\label{lea.m}
Let $-\Delta_g$ be the Laplace-Beltrami operator and $p$ be any point of $M$. In normal coordinates at $p$, we have
\begin{equation}
\label{2.expand}
(-\Delta_g)^mu=(-\Delta)^mu+\mathcal{D}^{2m}u+\mathcal{D}^{2m-1}u,	
\end{equation}
where $\mathcal{D}^{2m}$ is a linear differential operator of order $2m$ whose the coefficients are $O(|x-p|^2)$ as $x$ tends to $p$, while $\mathcal{D}^{2m-1}$ is a linear differential operator of order at most $2m-1$, and whose coefficients  belong to $C_{\mathrm{loc}}^l(\Rm)$ for all $l\geq 0$.
\end{lemma}

\begin{proof}
By the definition of Laplace-Beltrami operator, we have 
\begin{equation}
\label{a.m2}
-\Delta_gu=-\frac{1}{\sqrt{\mbox{det}g}}\partial_i(\sqrt{\mbox{det}g}g^{ij}\partial_ju),
\end{equation}
where $g^{ij}$ is the inverse of $g_{ij}.$ Using \eqref{a.m1} we can write
\begin{equation}
\label{a.m3}
-\Delta_gu=-g^{ij}\partial_{ij}u-\frac{\partial_i(\sqrt{\mbox{det}g}g^{ij})}{\sqrt{\mbox{det}g}}\partial_ju.
=-g^{ij}\partial_{ij}u-\vartheta_j\partial_ju,
\end{equation}
where $\vartheta_j$ is a smooth function. Based on \eqref{a.m3}, it is easy to see that in the final expression of $(-\Delta)_g^m$ the leading differential order  is 
$$g^{i_1j_1}g^{i_2j_2}\cdots g^{i_mj_m}\partial_{i_1j_1i_2j_2\cdots i_mj_m}\cdot,$$
where $i_a,j_b\in\{1,\cdots,2m\},~\forall a,b\in\{1,\cdots,m\}.$
While the remaining terms are order at most $2m-1$ and the coefficients are smooth due to the exponential map is differentiable with arbitrary order. Consider the leading term, using \eqref{a.m1}, we see that 
$$g^{ij}(x)=\delta_{ij}(x)+O(r^2),\quad \mathrm{if}~x\in B_r^M(p).$$
Therefore we can write 
$$\prod_{a=1}^mg^{i_aj_a}(x)=\prod_{a=1}^m\delta_{i_aj_a}+O(r^2),\quad \mathrm{if}~x\in B_r^M(p).$$
As a consequence, we can write
\begin{equation}
g^{i_1j_1}g^{i_2j_2}\cdots g^{i_mj_m}\partial_{i_1j_1i_2j_2\cdots i_mj_m}\cdot
=(-\Delta)^m\cdot+\mathcal{D}^{2m}\cdot,
\end{equation}
with $\mathcal{D}^{2m}$ satisfies the property stated in the lemma. Then we finish the proof.
\end{proof}

\medskip

\begin{remark} \label{rem-local}
\emph{Throughout the paper, when performing local computations, we may consider conformal normal coordinates, see \cite{cao} or \cite{LeeParker}, if needed. These are normal coordinates at a point $x_0$ for a metric $g_w=e^{2w}g$ with \emph{det}$(g_w)=1$ in a small neighborhood of $x_0$ and other useful properties, for which we refer the interested reader to \cite{WeinZhang}. Observe that the differential operator $P_g^{2m}$, after this change of the metric, can be still expanded as the right hand side of \eqref{2.expand}. Indeed, $w(x)=O(d_{g_w}^2(x,x_0))$ and it is smooth in a small neighborhood of $x_0$. Moreover, by \eqref{1.higher} we can write \eqref{1.mliouville} as 
$$
P_{g_w}^{2m}u=e^{-2mw}P_gu=e^{-2mw}\left(-Q_g^{2m}-\frac{\Lambda_m}{2|M|}\sum_{j=1}^N\alpha_j+\widetilde{Q}e^{2mu}\right),
$$
which is equivalent to
$$
e^{2mw}P_{g_w}^{2m}u=-Q_g^{2m}-\frac{\Lambda_m}{2|M|}\sum_{j=1}^N\alpha_j+\widetilde Qe^{2mu}.
$$
Concerning the differential operator $e^{2mw}P_{g_w}^{2m}$, it is known that the leading order operator is $e^{2mw}(-\Delta_{g_w})^m$. Then, by using the above Lemma \ref{lea.m} and the asymptotic behavior of $w(x)$, we can write
$$
e^{2mw}(-\Delta_{g_w})^{m}u=(-\Delta)^mu+\mathcal{P}^{2m}u+\mathcal{P}^{2m-1}u,
$$
with $\mathcal{P}^{2m}$ and $\mathcal{P}^{2m-1}$ satisfying the same properties as $\mathcal{D}^{2m}$ and $\mathcal{D}^{2m-1}$ in Lemma \ref{lea.m}.}
\end{remark}

\

In what follows, recall $\Lambda_m=(2m-1)!|\mathbb{S}^{2m}|$. We will need the following structural result on the Green's functional of the operators under consideration, see Lemma 2.1 in \cite{ndiaye}.
\begin{lemma}\label{expGreen}
Suppose $\mbox{Ker}\{P_g^{2m}\}=\{constants\}$. Then the Green's function $G(x,y)$ of $P_g^{2m}$ exists and has the following properties:
\begin{enumerate}
\item For all $u \in C^{2m}(M)$ we have for $x \neq y \in M$
$$
u(x)-\bar u=\int_M G(x,y)P^{2m}_g u(y)dV_g(y), \quad \int_M G(x,y) dV_g(y)=0,
$$
$$P_g^{2m}G(x,p)=\delta_p-\dfrac{1}{|M|},$$
where $\bar u$ is the average of $u$.
\item The function 
$$G(x,y)=H(x,y)+K(x,y)
$$
is smooth on $M\times M$, away from the diagonal. The function $K$ extends to a $C^{2,\alpha}$ function on $M \times M$ and $H$ satisfies
$$
H(x,y)=\frac{2}{\Lambda_m} \log\Big ( \frac{1}{r} \Big )\,f(r),
$$
\end{enumerate}
where $r$ is the geodesic distance from $x$ to $y$ and $f$ is a smooth positive, decreasing function such that $f(r)=1$ in a neighborhood of $r=0$ and $f(r)=0$ for $r \geq \text{inj}_g(M)$. 
\end{lemma}

As mentioned in the introduction, the total $Q$-curvature is a conformal invariant for which the following formula holds true.

\begin{theorem}\label{CGB}
Consider $D=\sum_{i=1}^N p_i \alpha_i$ where $p_i \in M$ and $\alpha_i >-1$. Let $g$ be a smooth metric on $M$ and $g_v=e^{2v}g$ be the conical metric representing $D$ as explained before \eqref{kappa}. Then, it holds
\begin{equation}
\label{CGBformula}
\int_MQ_{g_v}^{2m}\,dvol_{g_v}=\int_MQ_{g}^{2m}\,dvol_{g}+\frac{\Lambda_m}{2}\sum_{i=1}^N \alpha_i.
\end{equation}
\end{theorem} 
\begin{proof}
The proof is a standard argument (see e.g. \cite{FangMa} for $m=2$), using Lemmata \ref{lea.m} and \ref{expGreen}. See also \cite{BN} for a more general result which implies this statement as a particular case. 
\end{proof}

Finally, we state the general singular Adams-Trudinger-Moser inequality suitable to treat our problem. We focus here for simplicity on the case $P_g^{2m}\geq0$ and refer to the discussion in \cite{ndiaye} for the general case.
\begin{theorem}\label{ATM}
Consider $D=\sum_{i=1}^N p_i \alpha_i$ where $p_i \in M$ and $\alpha_i >-1$. Let $\widetilde Q>0$ be as in \eqref{qtilde}. Assume $P_g^{2m}\geq0$  and   $\mbox{Ker}\{P_g^{2m}\}=\{constants\}$. Then, there exists a constant $C=C(\alpha,M)$ such that for any $u \in H^{m}(M)$ we have
$$
\Lambda_m\Bigr(1+\min_j\left\{\a_j,0\right\}\Bigr)\log \int_M \widetilde Qe^{2m(u-\bar u)}\,dvol_{g} \leq m \int_M uP^{2m}_g u\,dvol_{g} + C,
$$
where $\bar u$ is the average of $u$.
\end{theorem}
\begin{proof}
The case without singularities is Proposition 2.2 in \cite{ndiaye}. The conic case follows by the same approach as in \cite{FangMa}.
\end{proof}

\

\section{Compactness property}
In this section we shall prove the compactness result of Theorem \ref{th1.1}. For simplicity of notation, there is no loss of generality to consider a blow-up sequence $u_k$ to
\begin{equation}
\label{2.u}
P_g^{2m}u_k+Q_g^{2m}= \widetilde Q e^{2mu_k},
\end{equation}
where
$$\widetilde Q=  Q_{g_{v}}^{2m}e^{-m\Lambda_m\sum_{j=1}^N\alpha_jG(x,q_j)}>0, \quad \Lambda_m=(2m-1)!|\mathbb{S}^{2m}|.$$
We call $p$ the blow-up point for the blow-up sequence $\{u_k\}$ if $u_k(p)\to+\infty$ as $k\to+\infty$. Collecting all the blow-up points into a set $\mathcal{B}$ and we name it the blow-up set for $\{u_k\}$. Theorem \ref{th1.1} will follow by showing a concentration phenomenon: 
$$\widetilde Q e^{2mu_k}\rightharpoonup\sum_{p\in\mathcal{B}}
(1+\alpha_p)\Lambda_m\delta_p\quad \mathrm{as}\quad k\to+\infty,$$
weakly in the sense of measures, 
\begin{equation*}
\alpha_p=\begin{cases}
0,\quad &\mbox{if}\quad p\notin\{q_1,\cdots,q_N\},\\
\alpha_j,\quad &\mbox{if}\quad p=q_j.
\end{cases}
\end{equation*}
It follows that when blow-up occurs, then necessarily 
$$\int_{M}\widetilde Qe^{2mu_k}\,dvol_g\to \sigma\in\Gamma\quad\mbox{as}\quad k\to +\infty,$$
where $\Gamma$ is given in \eqref{gamma}. 

\

First, we establish the following lemma.
\begin{lemma}
\label{le2.1}	
Let $\{u_k\}$ be a sequence of functions on $(M,g)$ satisfying \eqref{2.u}. Then for $i=1,\cdots,2m-1$ we have
\begin{equation}
\label{2.est-int}
\int_{B_r^M(x)}|\nabla^i u_k|^l\,dy\leq C(n)r^{2m-il},\quad 1\leq l<\frac{2m}{i},~\forall x\in M,~0<r<r_{inj}.
\end{equation}	
where $r_{inj}$ is the injectivity radius of $(M,g).$
\end{lemma}

\begin{proof}
Set $f_k:=\widetilde Qe^{2mu_k}-Q_g^{2m}$, which is bounded in $L^1(M)$. By Green's representation formula we have
\begin{equation}
\label{2.rep-u}
u_k(x)=\dashint_Mu_k\,dvol_g+\int_MG(x,y)f_k(y)\,dvol_g(y).		
\end{equation}	
For $x,y\in M,~x\neq y$, we have (see \cite[Lemma2.1]{ndiaye})
\begin{equation}
\label{2.est-g}
|\nabla^i_yG(x,y)|\leq\frac{C}{d_g(x,y)^i},\quad 1\leq i\leq 2m-1.
\end{equation} 
Then differentiating \eqref{2.rep-u} and using \eqref{2.est-g} and Jensen's inequality, we get
\begin{equation*}
\begin{aligned}
|\nabla^i u_k(x)|^l\leq~& 	C\left(\int_M\frac{1}{d_g(x,y)^i}|f_k(y)|\,dvol_g\right)^l\\
\leq~&C\int_M\left(\frac{\|f_k\|_{L^1(M)}}{d_g(x,y)^i}\right)^l
\frac{|f_k(y)|}{\|f_k\|_{L^1(M)}}\,dvol_g.
\end{aligned}
\end{equation*}	
From Fubini's theorem we conclude that
\begin{equation}
\label{2.estu_k}
\int_{B_r^M(x)}|\nabla^i u_k(x)|^l\,dvol_g\leq C\sup_{y\in M}\int_{B_r^M(x)}\frac{\|f_k\|_{L^1(M)}^l}{d_g(x,z)^{il }}\,dvol_g(z)	
\leq Cr^{2m-il}.
\end{equation}
It proves the lemma.
\end{proof}

Next, we shall give the minimal local mass around a blow-up point.

\begin{lemma}
\label{le2.2}
Let the sequence $u_k$ satisfy \eqref{2.u} and blowing-up at $q_j$. Suppose that
\begin{equation*}
\widetilde Qe^{2mu_k}\rightharpoonup\mathfrak{m},~\mbox{weakly in the sense of measures in M},
\end{equation*}
then 
$$\mathfrak{m}(q_j)\geq\frac12\min\{\Lambda_{m}(1+\alpha_j),\Lambda_{m}\}.$$
\end{lemma}

\begin{proof}
To show the thesis, it suffices to prove that if the following inequality holds
\begin{equation}
\label{2.bcon}
\int_{B^M(q_j,2r)}\widetilde Qe^{2mu_k}\,dvol_g<\frac12\min\{\Lambda_{m}(1+\alpha_j),\Lambda_{m}\},\quad r<\frac{r_{inj}}{2},
\end{equation}	
then 
\begin{equation}
\label{2.est}
u_k\leq C~\mathrm{in}~B^M(q_j,r).
\end{equation}	
We study equation \eqref{2.u} in terms of the local normal coordinates at $q_j$ (see for example Lemma \ref{lea.m} and Remark \ref{rem-local}). By the exponential map we define the pre-image of $B^M(q_j,r)$ by $B_r(0)$ and we use the same notation to denote $x\in M$ and its pre-image. We decompose $u_k$ as $u_k=u_{1k}+u_{2k}$, where 	$u_{1k}$ is the solution of 
\begin{equation}
\label{2.u1}
\begin{cases}
(-\Delta)^mu_{1k}=\widetilde Qe^{2mu_k}\Xi_r(x), &\mathrm{in}~B_{2r}(0),\\
u_{1k}=\Delta u_{1k}=\cdots=(-\Delta)^{m-1}u_{1k}=0,\quad &\mathrm{on}~\partial B_{2r}(0),
\end{cases}
\end{equation}
where $\Xi_r(x)=dvol_g(x)/dx=1+O(r^2)$ due to the metric tensor $g_{ij}(x)=\delta_{ij}+O(r^2).$ By \cite[Theorem 7]{m2}, we have 
\begin{equation}
\label{2.int-est-11}
e^{2m\ell |u_{1k}|}\in L^1(B_{2r}(0))\quad \mbox{for}~\ell\in\left(0,\frac{\Lambda_m}{2\|\widetilde Qe^{2mu_k}\Xi_r(x)\|_{L^1(B_{2r}(0))}}\right)
\end{equation}
and
\begin{equation}
\label{2.int-est}
\int_{B_{2r}(0)}e^{2m\ell|u_{1k}|}dx\leq C(p)r^{2m}.
\end{equation}
Let $G_{r}(x,y)$ be the Green's function of $(-\Delta)^m$ on $B_{2r}(0)$ satisfying the Navier boundary condition, i.e., 
\begin{equation*}
\begin{cases}
(-\Delta)^mG_r(x,y)=\delta_x(y), &\mathrm{in}~B_{2r}(0),\\	
G_r(x,y)=\cdots=\Delta^{m-1}G_r(x,y)=0,\quad&\mathrm{on}~\partial B_{2r}(0).	
\end{cases}		
\end{equation*}	
$G_r(x,y)$ can be decomposed as 
\begin{equation*}
G_r(x,y)=-\frac{2}{\Lambda_m}\log|x-y|+R_r(x,y)
\end{equation*}	
with $R_r(x,y)$ a smooth function for $x,y\in B_{2r}.$ By the Green's representation formula we have
\begin{equation}
\label{2.rep}
u_{1k}(x)=-\frac{2}{\Lambda_m}\int_{B_{2r}(0)}\log|x-y|\widetilde Qe^{2mu_k}\Xi_r(y)dy+O(1),\quad x\in B_{3r/2}(0).
\end{equation}
Observe that 
$$
\widetilde Q= d_g(x,q_j)^{2m\alpha_j} \widehat Q
$$
where 
\begin{equation}
\label{2.q}
\begin{aligned}
\widehat Q=Q_{g_v}^{2m}e^{-m\Lambda_m\alpha_jR(x,q_j)-m\Lambda_m\sum_{i\neq j}^N\alpha_iG(x,q_i)}~\mbox{is a smooth function in}~B_{2r}(0).
\end{aligned}
\end{equation}
On the other hand, by using $G(x,y)$ and the Green's representation formula we get that
\begin{equation}
\label{2.repm}
\begin{aligned}
u_k(x)=~&u_{1k}(x)+u_{2k}(x)\\
=~&\overline{u}_k+\int_{M}
G(x,y)\widetilde Qe^{2mu_k}\,dvol_g-\int_MG(x,y)Q_g^{2m}dvol_g,
\end{aligned}
\end{equation}	
where $\overline{u}_k$ is the average of $u_k$. Since it is known that the leading term of $G$ and $G_{2r}$ carry the same singular behavior, we get from Jensen's inequality that
\begin{equation}
\label{2.u2}
u_{2k}=\overline{u}_k+O(1)=\overline{v}_k+O(1)\leq \log\dashint_{M}e^{v_k}+O(1)=O(1),
\end{equation}	
where we used the average of Green function on $M$ is zero. Therefore, we conclude that $u_{2k}$ is bounded uniformly from above. Next, we shall prove that $u_{1k}$ is bounded and our discussion is separated into two cases:

\

\noindent Case 1. $\alpha_j>0$, then \eqref{2.bcon} is equivalent to 
\begin{equation*}
\int_{B^M_{2r}(q_j)}\widetilde Qe^{2mu_k}\,dvol_g<\frac12\Lambda_m.
\end{equation*}	
Since $u_{2k}$ is bounded from above in $B_{3r/2}(0)$, we see from  \eqref{2.int-est} that there exists some $\ell>1$ such that $\widetilde Qe^{2mu_k}\in L^\ell(B_{2r}^M(q_1))$. It is easy to see that $u_k\in L^1(M)$. Together with \eqref{2.rep} we can easily see that $u_{1k},u_{2k}\in L^1(B_{3r/2}(0)).$ By the interior regularity results in \cite[Theorem 1]{b} we get that
\begin{equation}
\label{2.est-b}	
\|u_{1k}\|_{W^{2m,\ell}(B_{r}(0))}\leq \|\widetilde Qe^{2mu_k}\|_{L^\ell(B_{3r/2}(0))}+\|u_{1k}\|_{L^{1}(B_{r}(0))}\leq C.	
\end{equation}
Thus by the classical Sobolev inequality we get that $u_{1k}\in L^\infty(B_r(0))$. 

\medskip

\noindent Case 2. If $\alpha_j\in(-1,0)$ then \eqref{2.bcon} is equivalent to 
\begin{equation*}
\int_{B_{2r}^M(q_j)}\widetilde Qe^{2mu_k}\,dvol_g<\frac12\Lambda_m(1+\alpha_j).
\end{equation*}	
It is not difficult to see $|x|^{2m\alpha_j}\in L^{\ell}(B_{2r})$ for any $\ell\in[1,-\frac{1}{\alpha_j})$ and $e^{2mu_{1k}}\in L^p(B_{2r})$ for $p\in[1,\frac{1}{1+\alpha_1}+\epsilon)$ for some small strictly positive number $\epsilon$ by \eqref{2.int-est-11}. As a consequence, we get that $|x|^{2m\alpha_1}e^{2mu_{1k}}\in L^{\ell}(B_{2r}(0))$ for some $\ell>1$ by H\"older inequality. Repeating the arguments as in Case 1,  we obtain that $u_{1k}$ is bounded uniformly in $B_r(0)$. 

After establishing that $u_{1k}$ is bounded in $B_r(0)$, combining with \eqref{2.u2} we derive that $u_k$ is bounded above in $B_r^M(q_j)$. Then we finish the proof of this lemma.	
\end{proof}

\medskip

We shall derive now the quantization result and the concentration property of the bubbling solution.

\begin{proposition}
Let $\{u_k\}$ be a sequence of solutions to \eqref{2.u} and $\mathcal{B}$ be its blow-up set, then we have the following convergence in the sense of measures
\begin{equation}
\widetilde Q e^{2mu_k}\rightharpoonup \sum_{p\in\mathcal{B}}\Lambda_m(1+\alpha_p)\delta_p,\quad \mbox{as}~k\to+\infty,
\end{equation}
where
\begin{equation*}
\alpha_p=\begin{cases}
\alpha_i,\quad &\mbox{if}\quad p=q_i\in\{q_1,\cdots,q_N\},\\	
0,\quad  &\mbox{if}\quad  p\in \mathcal{B}\setminus\{q_1,\cdots,q_N\}	.
\end{cases}
\end{equation*}
In particular, $u_k\to-\infty$ uniformly on any compact subset of $M\setminus \mathcal{B}.$
\end{proposition}

\begin{proof}
For any compact set $K\subset M\setminus\mathcal{B}$, we can use the Green's representation formula
\begin{equation}
\label{2.p-1}
u_k(x)-u_k(y)=\int_{M}\left(G(x,z)-G(y,z)\right)\left(\widetilde Qe^{2mu_k}-Q_{g}^{2m}\right)dz
\end{equation}
together with the estimate \eqref{2.est-g} to derive that
\begin{equation}
\label{2.p-2}
|\nabla^i u_k(x)|\leq C(K)~\mbox{for}~x\in K,\quad  1\leq i\leq 2m-1.
\end{equation}	
Then from equation \eqref{2.u} and classical elliptic estimates  we get the $2m$ order derivatives of $u_k$
\begin{equation}
\label{2.p-3}
|\nabla^{2m}u_k(x)|\leq C(K)~\mbox{for}~x\in K.
\end{equation}

To proceed with our discussion we introduce the following quantity 
\begin{equation}
\label{2.p-l}
\sigma_p=\lim_{r\to0}\lim_{k\to+\infty}\int_{B^M(p,r)}\widetilde Qe^{2mu_k}\,dvol_g.
\end{equation}
It has been shown in Lemma \ref{le2.2} that $\sigma_p$ has a positive lower bound at the blow-up point. From the fact that $\int_M Q_g^{2m}\,dvol_g$ is finite, we conclude that the blow-up points are finitely-many. At a regular blow-up point $p$ it has been already shown in \cite[Theorem 2]{m1} that
\begin{equation*}
\widetilde Qe^{2mu_k}\rightharpoonup \Lambda_m\delta_p \quad\mathrm{in}\quad B^M_{r_p}(p),
\end{equation*} 
where $r_p$ is chosen such that $B^M(p,r_p)\cap(\mathcal{B}\setminus\{p\})=\emptyset$. In the following discussion  we will focus on the singular blow-up point. Without loss of generality, we shall consider $u_k$ in $B^M_{2r}(q_1)$, where $r$ is chosen such that $B^M_{2r}(q_1)$ only contains $q_1$ from $\mathcal{B}$. We first claim that
\begin{equation}
\label{2.p.claim}
u_k\to-\infty\quad\mbox{for}\quad x\in B_{2r}^M(q_1)\setminus\{q_1\}.
\end{equation}
We prove it by contradiction. Suppose that $u_k$ is uniformly bounded below at some point away from $q_1$. Then by \eqref{2.p-2} we derive that
\begin{equation}
\label{2.p.conv}
u_k\to u_0\quad\mbox{in}\quad C_{\mathrm{loc}}^{2m-1,\sigma}(B^M_{2r}(q_1)\setminus\{q_1\}),	\quad \sigma\in(0,1),
\end{equation}	
with the limit function verifying 
\begin{equation}
P_g^{2m}u_0+Q_g^{2m}=d_g(x,q_1)^{2m\alpha_1}\widehat Qe^{2mu_0}\quad \mathrm{in}\quad B^M_{2r}(q_1)\setminus \{q_1\},
\end{equation}
where $\widehat Q$ is a smooth function around $q_1$ defined analogously as in \eqref{2.q} and $d_g(x,q_1)$ denotes the geodesic distance between $x$ and $q_1$ with respect to the metric $g$. According to the definition of $\sigma_p$ (see \eqref{2.p-l}) we see that $u_0$ satisfies 
\begin{equation*}
P_g^{2m}u_0+Q_g^{2m}=d_g(x,q_1)^{2m\alpha_1}\widehat Qe^{2mu_0}+\sigma_{q_1}\delta_{q_1}\quad \mbox{in}\quad B_{2r}^M(q_1).
\end{equation*} 
Using the Green's representation formula for $u_0$, we have
\begin{equation}
\label{2.p-re}
u_0(x)=\sigma_{q_1}G(x,q_1)+v_0(x),
\end{equation}	
where 
\begin{equation}
\label{2.p-v}
\begin{aligned}
v_0(x)=&\int_M\frac{2}{\Lambda_m}\log d_g(x,y)\left(d_g(y,q_1)^{2m\alpha_1}\widehat Qe^{2mu_0}\right)\,dvol_g\\
&+\int_MR(x,y)\left(d_g(y,q_1)^{2m\alpha_1}\widehat Qe^{2mu_0}\right)\,dvol_g\\
&+\bar u_0-\int_MG(x,y)Q_g^{2m}dvol_g.
\end{aligned}
\end{equation}	
Denoting the two terms on the right hand side by $\hat v_1$ and $\hat v_2$ respectively, it is not difficult to see that $\hat v_2$ is smooth. In the following we shall prove that $\hat v_1(x)$ is bounded in $x\in B^M_r(q_1)$. In fact, for $x\in B^M_r(q_1)$ we have 	
\begin{equation}
\label{2.p-low}
\begin{aligned}
v_0(x)~&=\int_{B^M_r(q_1)}G(x,y)d_g(y,q_1)^{2m\alpha_1}\widehat Qe^{2mu_0}\,dvol_g+O(1)\\
&\geq\frac{2}{\Lambda_m}\log\frac{1}{r}\|\widehat Q|x|^{2m\alpha_1}e^{2mu_0}\|_{L^1(B^M_{2r}(q_1))}+O(1).
\end{aligned}
\end{equation}
This provides a lower bound for $v_0(x)$. On the other hand, we have
\begin{equation}
\label{2.p.lb}
d_g(x,q_1)^{2m\alpha_1}e^{2mu_0}
\geq Cd_g(x,q_1)^{2m(\alpha_1-\frac{2\sigma_{q_1}}{\Lambda_m})}.
\end{equation}
Using the fact that the left hand side of \eqref{2.p.lb} is integrable we get 
\begin{equation}
\label{2.ine-alpha}
\alpha_1-\frac{2\sigma_{q_1}}{\Lambda_m}>-1.
\end{equation}
When $\alpha_1<0$, we see that the above inequality \eqref{2.ine-alpha} implies that $\sigma_{q_1}<\frac12\Lambda_m(1+\alpha_1)$. Then $u_k$ can not blow-up at $q_1$ by Lemma \ref{le2.2} and we get a contradiction. This implies \eqref{2.p.claim} for $\alpha_1<0$. For $\alpha_1>0$ we have 
\begin{equation}
\label{2.p.lb-2}
Cd_g(x,q_1)^{2m\alpha_1-\frac{4m\sigma_{q_1}}{\Lambda_m}}
e^{v_0(x)}
\geq d_g(x,q_1)^{2m\alpha_1}e^{2mu_0}\geq 
Cd_g(x,q_1)^{2m\alpha_1-\frac{4m\sigma_{q_1}}{\Lambda_m}}.
\end{equation}	
In order to show that $\hat v_1(x)$ is bounded in $B_r^M(q_1)$, we study $\hat v_1(x)$ in terms of local coordinates at $q_1$. Then $d_g(x,q_1)$ can be regarded as $|x^p-0|$, where $x^p$ denotes its pre-image of $x$ under the exponential map at $q_1$. By a little abuse of notation, we still denote $x^p$ by $x$. Then we notice that $\hat v_1(x)$ satisfies  
\begin{equation}
\label{2.p.eq}
(-\Delta)^{m}\hat v_1(x)=|x|^{2m\alpha_1}\widehat Qe^{2mu_0}\Xi(x)\quad\mbox{in}\quad B_{2r}(0),
\end{equation}
where $\Xi(x)=\frac{dvol_g(x)}{dx}$ is bounded above and below in $B_{2r}(0)$ since the metric tensor is comparable to the standard Euclidean metric. Since we only consider the local behavior of $\hat v_1(x)$ in $B_r(0)$, by multiplying a cut-off function $\chi(x)$ with $\chi(x)=1$ for $|x|\leq r$ and $\chi(x)=0$ for $|x|\geq 2r$, we have $\tilde v_1(x):=\chi(x)\hat v_1(x)$ verifies   
\begin{equation}
\label{2.eq-v1}
\begin{cases}
(-\Delta)^m\tilde v_1(x)=|x|^{2m\alpha_1}\widehat Qe^{2mu_0}\Xi(x)+\Xi_0(x)
\quad  &\mbox{in}\quad  B_{2r}(0),\\
\\
\tilde v_1(x)=\Delta \tilde v_1(x)=\cdots=\Delta^{m-1}\tilde v_1(x)=0 \quad  &\mbox{on}\quad \partial B_{2r}(0),
\end{cases}
\end{equation}	
where $\Xi_0(x)$ is smooth in $B_{2r}(0)$. It is not difficult to see that 
$$|x|^{2m\alpha_1}\widehat Qe^{2mu_0}\Xi(x)+\Xi_0(x)\in L^1(B_2r(0)).$$ 
We decompose $|x|^{2m\alpha_1}\widehat Qe^{2mu_0}\Xi(x)+\Xi_0(x)$ into $P_1(x)$ and $P_2(x)$ with 
\begin{equation}
\label{2.p-condition}
\|P_1\|_{L^1(B_{2r}(0))}\leq \varepsilon\quad \mathrm{and}\quad P_2\in L^\infty(B_{2r}(0)).
\end{equation}
Correspondingly we decompose $\tilde v_1$ into $\tilde v_{11}$ and $\tilde v_{12}$, where $\tilde v_{1i},~i=1,2$ solve
\begin{equation}
\label{2.eq-v1i}
\begin{cases}
(-\Delta)^m\tilde v_{1i}(x)=P_i(x)
\quad  &\mbox{in}\quad B_{2r}(0),\\
\\
\tilde v_{1i}(x)=\Delta \tilde v_{1i}(x)=\cdots=\Delta^{m-1}\tilde v_{1i}(x)=0 \quad  &\mbox{on}\quad \partial B_{2r}(0).
\end{cases}
\end{equation}	
For $\tilde v_{11},$ by \cite[Theorem 7]{m2} we get that 
$e^{\frac{\Lambda_m}{2\e}\tilde v_{11}}\in L^1(B_{2r}(0)).$
While for $\tilde v_{12}$, using the classical elliptic regularity theory we derive that $\tilde v_{12}\in L^\infty(B_{2r}(0))$. Together with  \eqref{2.p.lb-2} we can select $\varepsilon$ sufficiently small such that 
$$|x|^{2m\alpha_1}\widehat Qe^{2mu_0}\Xi(x)\in L^l(B_{2r}(0))\quad \mbox{for some}\quad l>1.$$
Returning to equation \eqref{2.eq-v1} we apply the regularity theory to deduce that $\hat v_1\in W^{2m,l}(B_{2r}(0))$. As a consequence, we have $v_0\in W^{2m,l}(B_{2r}(0))$ and it implies that $|v_0|\leq  C$ for some constant $C$ in $B_{r}(0)$ by the classical Sobolev inequality. Thus we have proved $v_0$ is bounded in $B^M_r(q_1)$. Together with \eqref{2.p.lb-2} we get that
\begin{equation}
\label{2.p-bb}
d_g(x,q_1)^{2m\alpha_1}\widehat Qe^{2mu_0}
\sim{d_g(x,q_1)^{2m(\alpha_1-\frac{2\sigma_{q_1}}{\Lambda_m})}}\quad \mbox{if}\quad  \alpha_1>0.
\end{equation}
Next we shall derive a contradiction by making use of the Pohzozaev identity. It is known that equation \eqref{2.u} can be written as
\begin{align}
\label{2.i-0}
(-\Delta_g)^{m}u_k(x)+\mathcal{A}u_k+Q_g^{2m}=\widetilde Qe^{2mu_k}\quad \mathrm{in}\quad B_r^M(q_1),
\end{align}	
where $\mathcal A$ is a linear differential operator of order at most $2m-1$, moreover the coefficients of $A$ belong to $C_{\mathrm{loc}}^l(M)$ for all $l\geq 0$. Using the local normal coordinates, by Lemma \ref{lea.m} (see also Remark \ref{rem-local}) we could write \eqref{2.i-0} as
\begin{equation}
\label{2.i-1}
(-\Delta)^mu_k+\mathcal{D}^{2m}u_k+\mathcal{C}u_k+Q_g^{2m}=\widetilde Qe^{2mu_k}\quad \mathrm{in}\quad B_r(0),
\end{equation}
where $\mathcal{D}^{2m}$ is a linear differential operator of order $2m$ and the coefficients are of order $O(|x|^2)$ with its derivative of arbitrary order smooth, while $\mathcal{C}$ is a linear differential operator of order at most $2m-1$, and the coefficients of $B_k$ belong to $C_{\mathrm{loc}}^l(\Rm)$ for all $l\geq 0$. Multiplying by $x\cdot\nabla u_k$ on both sides, concerning the right hand side, we have
\begin{equation}
\label{2.i-r}
\begin{aligned}
\mbox{R.H.S. of }\eqref{2.i-1}=~&\frac{1}{2m}\int_{B_r(0)}x\cdot\nabla (|x|^{2m\alpha_1}\widehat Qe^{2m u_k})\,dx
-\alpha_1\int_{B_r(0)}|x|^{2m\alpha_1}\widehat Qe^{2m u_k}\,dx\\
&-\frac{1}{2m}\int_{B_r(0)}(x\cdot\nabla\widehat Q)|x|^{2m\alpha_1}e^{2m u_k}\,dx\\
=~&\frac{1}{2m}\int_{\partial B_r(0)}|x|^{2m\alpha_1+1}\widehat Qe^{2m u_k}\,ds
-(\alpha_1+1)\int_{B_r(0)}|x|^{2m\alpha_1}\widehat Qe^{2m u_k}\,dx\\
&-\frac{1}{2m}\int_{B_r(0)}(x\cdot\nabla\widehat Q)|x|^{2m\alpha_1}e^{2m u_k}\,dx\\
\to~&-(1+\alpha_1)\sigma_{q_1}+o_r(1)\quad\mathrm{as}\quad  k\to+\infty.
\end{aligned}	
\end{equation}
Next, we consider the left hand side of \eqref{2.i-1}. At first, for the fourth term, we have
\begin{equation}
\label{2.i-l-3}
\left|\int_{B_r(0)}\langle x,\nabla u_k\rangle Q_g^{2m}dx\right|
\leq r\int_{B_r(0)}|\nabla u_k|dx\leq Cr^{2m},
\end{equation}
where we used Lemma \ref{le2.1}. Therefore
\begin{equation}
\label{2.i-l-3-1}
\lim_{r\to0}\lim_{k\to+\infty}\left|\int_{B_r(0)}\langle x,\nabla u_k\rangle Q_g^{2m}dx\right|=0.
\end{equation}
For the third term, we have
\begin{equation}
\label{2.i-l1}
\begin{aligned}
\left|\int_{B_r(0)}\langle x,\nabla u_k\rangle\mathcal{C}u_k\right|	
\leq ~&C\sum_{i=0}^{2m-1}\int_{B_r(0)}|x||\nabla^iu_k||\nabla u_k|dx\\
\leq~&\sum_{i=1}^{2m-1}
\left(\int_{B_r(0)}|x|^{s_i}|\nabla u_k|^{s_i}dx\right)^{\frac{1}{s_i}}\left(\int_{B_r(0)}|\nabla^iu_k|^{t_i}dx\right)^{\frac{1}{t_i}}\\
&+C\int_{B_r(0)}|\nabla u_k|dx,	
\end{aligned}
\end{equation}
where we used $|x| |u_k|\leq C$ in $B_r(0)$,
$$t_i=\frac{2m}{i}-\delta\quad\mbox{and}\quad
s_i=\frac{2m-\delta i}{2m-i-\delta i},\qquad \delta\in\left(0,\frac{1}{2(2m-1)}\right).$$
From \eqref{2.p-re}, \eqref{2.p-v} and since $\hat v_1,~\hat v_2$ are bounded from above, we can see that $|x\cdot\nabla u_0|\leq C$ in $B_r(0)$. Together with Lemma \ref{le2.1} we have 
\begin{equation}
\label{2.i-l2}
\int_{B_r(0)}|x|^{s_i}|\nabla u_0|^{s_i}dx\leq Cr^{2m},
\quad\mbox{and}\quad
\begin{aligned}
\int_{B_r(0)}|\nabla^iu_0|^{t_i}dx\leq Cr^{i\delta}.
\end{aligned}
\end{equation}
As a consequence of \eqref{2.i-l1} and \eqref{2.i-l2}, we see that
\begin{equation}
\label{2.i-l-2}
\lim_{r\to0}\lim_{k\to+\infty}\left|\int_{B_r(0)}\langle x,\nabla u_k\rangle\mathcal{C}u_k\right|\,dx=0.
\end{equation}
For the second term, we have already seen that $v_0\in W^{2m,l}(B_r(0))$ for some $l>1$, then using \eqref{2.p-re}, we see that
\begin{equation*}
\int_{B_r(0)}x\cdot\nabla u_0 \mathcal{D}^{2m}u_0\,dx
\leq C\int_{B_r(0)}\frac{1}{|x|^{2m-2}}\,dx+Cr^2\|v_0\|_{W^{2m,1}(B_{r}(0))}\leq Cr^2.
\end{equation*}
It leads to 
\begin{equation}
\label{2.i-l-22}
\lim_{r\to0}\lim_{k\to+\infty}\left|\int_{B_r(0)}\langle x,\nabla u_k\rangle\mathcal{D}^{2m}u_k\,dx\right|=0.
\end{equation}
Therefore, from \eqref{2.i-l-3-1}, \eqref{2.i-l-2} and \eqref{2.i-l-22} we get that except from the first term on the left hand side of \eqref{2.i-1}, the other terms vanish in the limit. It remains to study the term $\int_{B_r(0)}(-\Delta)^mu_k x\cdot\nabla u_kdx$. We shall only consider the case when $m$ is even, the argument for the case $m$ is odd goes almost the same. We set $m=2m_0$. Using the Pohozaev identity \eqref{a.poho1}, replacing $f$ by $\frac{2\sigma_{q_1}}{\Lambda_m}\log|x|$ plus a smooth function, after direct computations we get that
\begin{equation}
\label{2.il-1}
\begin{aligned}
&\int_{B_r(0)}(-\Delta)^mu_0\langle x,\nabla u_0\rangle dx\\
&=\sum_{i=2}^{m_0}\int_{\partial B_r(0)}2^{2m}(m-1)!(m-1)!\left(1-\frac{i-1}{m-i}\right)\frac{\sigma_{q_1}^2}{\Lambda_m^2}\frac{1}{r^{2m-1}}ds\\
&\quad +\sum_{i=1}^{m_0}\int_{\partial B_r(0)}2^{2m}(m-1)!(m-1)!\left(\frac{i-1}{m-i}-1\right)\frac{\sigma_{q_1}^2}{\Lambda_m^2}\frac{1}{r^{2m-1}}ds\\
&\quad -\int_{\partial B_r(0)}2^{2m-1}(m-1)!(m-1)!\frac{\sigma_{q_1}^2}{\Lambda_m^2}\frac{1}{r^{2m-1}}ds+o_r(1)
\\
&\to -2^{2m-1}(m-1)!(m-1)!\frac{\sigma_{q_1}^2}{\Lambda_m^2}|\mathbb{S}^{2m-1}|\quad \mathrm{as}\quad r\to0.
\end{aligned}
\end{equation} 
It is known that we can write 
\begin{equation*}
\Lambda_m=(2m-1)!|\mathbb{S}^{2m}|=2^{2m-1}(m-1)!(m-1)!|\mathbb{S}^{2m-1}|.
\end{equation*}
Together with \eqref{2.i-r} and \eqref{2.il-1} we derive that
\begin{equation}
\label{2.i-final}
(1+\alpha_1)\sigma_{q_1}=\frac{\sigma_{q_1}^2}{\Lambda_m}.
\end{equation}
Recalling also Lemma \ref{le2.2} we get
\begin{equation}
\label{2.i-final-1}
\sigma_{q_1}=(1+\alpha_1)\Lambda_m.
\end{equation}
Returning to equation \eqref{2.p-bb} we see that
\begin{equation}
d_g(x,q_1)^{2m\alpha_1}\widehat Qe^{2mu_0}
\sim d_g(x,q_1)^{-2m-2m(1+\alpha_1)}\geq d_g(x,q_1)^{-2m},
\end{equation}	
which contradicts  	$d_g(x,q_1)^{2m\alpha_1}\widehat Qe^{2mu_0}\in L^1(B_r(0))$. Therefore, $u_k\to-\infty$ uniformly on any compact subset of $B^M_{2r}(q_1)\setminus\{q_1\}$. 

\medskip

It remains to show the quantization $\sigma_{q_1}$ is exactly $\Lambda_m(1+\alpha_1)$. We consider the function $\hat u_k(x)=u_k(x)-c_k$, with $c_k=\dashint_{\partial B_r(0)}u_k(x)\to-\infty$. As before, we can show that $\hat u_k(x)$ converges to some function $\hat u_0(x)$ in $C_{\mathrm{loc}}^{2m}(B_{2r}\setminus\{0\})$ and we can write
$$\hat u_0(x)=-\frac{2\sigma_{q_1}}{\Lambda_m}\log d_g(x,q_1)+\hat v(x).$$
Repeating the previous argument, again by the Pohozaev identity we derive that $\sigma_{q_1}=\Lambda(1+\alpha_1)$ and we finish the proof.
\end{proof}

\

\section{Existence result}

In this section we are going to prove the existence and multiplicity results of Theorems \ref{th1.2} and \ref{th-mult}. To make the exposition more transparent we assume hereafter for simplicity $P_g^{2m}\geq0$. The general case can be treated by suitable adaptations, see Remark \ref{rem-general}.

\medskip
	
Solutions of \eqref{1.mliouville} are critical points of the functional
\begin{equation}\label{functional}
\begin{aligned}
\mathcal E(u)=~&2m \int_M uP^{2m}_gu\, dvol_g +4m \int_M \Bigr(Q^{2m}_g+\dfrac{\Lambda_m}{2|M|} \sum_{j=1}^N \alpha_j  \Bigr)u\,dvol_g\\
&-2\kappa_{g_v} \log \int_{M} \widetilde Qe^{2mu}\,dvol_g
\end{aligned}
\end{equation}
with $u \in H^{m}(M)$, where we recall
$$
	\widetilde Q=  Q_{g_{v}}^{2m}e^{-m\Lambda_m\sum_{j=1}^N\alpha_jG(x,q_j)}>0, \quad \Lambda_m=(2m-1)!|\mathbb{S}^{2m}|
$$
and 
$$
  \kappa_{g_v}=\int_M\widetilde Qe^{2mu}\,dvol_g=\int_MQ_{g}^{2m}\,dvol_{g}+\dfrac{\Lambda_m}{2} \sum_{j=1}^N \alpha_j.
$$
We point out we consider here 
$$
\alpha_j>0 \quad \forall j=1,\dots,N.
$$ 
Then, in particular, the Adams-Trudinger-Moser inequality in Theorem \ref{ATM} implies the functional $\mathcal E$ is coercive and bounded from below provided $\kappa_{g_v}<\Lambda_m$. Thus, existence of solutions in this subcritical case follows by direct method of calculus of variations.

\medskip

In the supercritical case $\kappa_{g_v}>\Lambda_m$ the functional is unbounded from below and we need to apply a min-max method based on the  topology of the sublevels of the functional
$$
	\mathcal E^a =\bigr\{ u\in H^m(M) \,:\, \mathcal E(u)\leq a \bigr\}.
$$ 
The rough idea is that the low sublevels carry some non-trivial topology while the high sublevels are contractible and such change of topology jointly with the compactness property of Theorem~\ref{th1.1} (provided $\kappa_{g_v}\notin\Gamma$) detects a critical point. The main step is the study of low sublevels which is done by improved version of the Adams-Trudinger-Moser inequality and suitable test functions.

\medskip

Let us start by pointing out that once the Adams-Trudinger-Moser inequality is available (Theorem~\ref{ATM}), then a standard argument yields improved versions of it under a spreading of the conformal volume $\widetilde Qe^{2mu}$ in, say, $l$ disjoint subsets as expressed in \eqref{spread}. Somehow, it is possible to sum up localized versions of the inequality which are in turn based on cut-off functions and improve the Adams-Trudinger-Moser constant to roughly $l\Lambda_m$. We refer the interested readers for example to Lemma 4.1 in \cite{ndiaye} and references therein.
\begin{lemma} \label{lem-improved}
Let $\delta,\theta>0$, $l\in\mathbb{N}$ and $\Omega_1,\dots,\Omega_l\subset M$ be such that $d(\Omega_i,\Omega_j)>\delta$ for any $i\neq j$. Then, for any $\e>0$ there exists $C=C(\e,\delta,\theta,L,M)$ such that for any $u\in H^m(M)$ such that
\begin{equation} \label{spread}
	\int_{\Omega_i}\dfrac{\widetilde Qe^{2mu}\,dvol_g}{\int_M \widetilde Q\,e^{2mu}\,dvol_g} \geq \theta \qquad \forall i\in\{1,\dots, l\}
\end{equation}
it holds
$$
l\Lambda_m\log \int_M \widetilde Qe^{2m(u-\bar u)}\,dvol_{g} \leq (1+\e)m \int_M uP^{2m}_g u\,dvol_{g} + C,
$$
where $\bar u$ is the average of $u$.
\end{lemma}

Improved inequalities then yield lower bounds on the functional $\mathcal E$. As a consequence, in the low sublevels $\widetilde Qe^{2mu}$ has to be concentrated in not too many different subsets, as shown in the following result. 
\begin{lemma}\label{lem-bound}
Suppose $\kappa_{g_v}<(k+1)\Lambda_m$ for some $k\in\mathbb{N}$. Then, $\forall\e,r>0$ there exists $L=L(\e,r)\gg1$ such that $\forall u\in\mathcal E^{-L}$ there exist $k$ points $\{p_1,\dots,p_k\}\subset M$ such that
$$
	\int_{\cup_{i=1}^k B_r^M(p_i)}\dfrac{\widetilde Qe^{2mu}\,dvol_g}{\int_M \widetilde Q\,e^{2mu}\,dvol_g} \geq 1-\e.
$$
\end{lemma}
\begin{proof}
We sketch here the proof. Suppose the thesis is false. Then, using a standard covering argument as in Lemma 2.3 of \cite{dm} there exist $k+1$ disjoint subsets $\Omega_1,\dots,\Omega_{k+1}\subset M$ in which $\widetilde Qe^{2mu}$ is spread in the sense of \eqref{spread}. Therefore, applying the improved Adams-Trudinger-Moser inequality of Lemma \ref{lem-improved} we would get a lower bound of the functional
$$
	\mathcal E(u)\geq 2m\left( 1-\dfrac{\kappa_{g_v}}{(k+1)\Lambda_m}(1+\e) \right)\int_M uP^{2m}_g u\,dvol_{g} + {l.o.t.}
$$
By assumption $\kappa_{g_v}<(k+1)\Lambda_m$ and hence we can take a sufficiently small $\e>0$ such that 
$$
1-\frac{\kappa_{g_v}}{(k+1)\Lambda_m}(1+\e)\geq0
$$
which yields $\mathcal E(u)\geq-L$ for some $L\gg1$. This is not possible since we were considering $u\in\mathcal E^{-L}$.
\end{proof}

\medskip

It is then convenient to describe the low sublevels by means of formal barycenters of $M$ of order $k$, that is unit measures supported in at most $k$ points on $M$, defined by
\begin{equation} \label{bary}
	M_k = \left\{ \sum_{i=1}^k t_i\delta_{p_i} \,:\, \sum_{i=1}^k t_i=1,\,t_i\geq0,\,p_i\in M, \ \forall i=1,\dots,k \right\}.
\end{equation} 
The idea is to use a projection within unit measures such that
$$
\dfrac{\widetilde Qe^{2mu}\,dvol_g}{\int_M \widetilde Q\,e^{2mu}\,dvol_g} \mapsto \sigma\in M_k .
$$
This is done exactly as in Proposition 3.1 of \cite{dm} by using Lemma \ref{lem-bound} to get the following.
\begin{proposition}
Suppose $\kappa_{g_v}<(k+1)\Lambda_m$ for some $k\in\mathbb{N}$. Then, there exist $L\gg1$ and a projection $\Psi:\mathcal E^{-L}\to M_k$. 
\end{proposition}
Recall now that we are assuming there exists a retraction $R:M\to \M$ with $\M\subset M$ a closed $n$-dimensional submanifold, $n\in[1,2m]$, such that $\a_j\notin\M$ for all $j=1,\dots,N$. Let $\M_k$ be the set of formal barycenters of $\M$. We can then define a map $\Psi_R:\mathcal E^{-L}\to \M_k$ simply by considering the composition
$$
	\mathcal E^{-L} \stackrel{\Psi}{\longrightarrow} M_k \stackrel{R_*}{\longrightarrow}  \M_k,
$$
where $R_*$ is the push-forward of measures induced by the retraction $R$. Therefore, we have the following result.
\begin{lemma} \label{lem-psi}
Suppose $\kappa_{g_v}<(k+1)\Lambda_m$ for some $k\in\mathbb{N}$. Then, there exist $L\gg1$ and a continuous map $\Psi_R:\mathcal E^{-L}\to \M_k$.
\end{lemma}

\medskip

The low sublevels are thus naturally described (at least partially) by $\M_k$. As a matter of fact, we are going to construct a reverse map, mapping continuously $\M_k$ into $\mathcal E^{-L}$. This is done by suitable test functions on which the functional attains low values. The idea here is that, since $\M$ does not contain conical points $q_j$, we may consider a family of \emph{regular} bubbles centered on $\M$. We thus take a non-decreasing cut-off function $\chi_\delta$ such that
$$
\begin{cases}
	\chi_\delta(t)=t, \quad &t\in[0,\delta], \\
	\chi_\delta(t)=2\delta, \quad &t\geq2\delta,
\end{cases}
$$
let $\lambda>0$ and then define $\Phi:\M_k\to H^m(M)$ by
$$
	\Phi(\sigma)=\varphi_{\lambda,\sigma}, \quad \sigma=\sum_{i=1}^k t_i\delta_{p_i}\in\M_k,
$$
where
$$
	\varphi_{\lambda,\sigma}(y)=\frac{1}{2m}\log\sum_{i=1}^k t_i \left( \dfrac{2\lambda}{1+\lambda^2\chi_\delta^2(d(y,p_i))} \right)^{2m}.
$$
Now, since we are considering bubbles centered on $\M$ which does not contain conical points, we can neglect the effect of the singularities and all the following estimates are carried out exactly as in the regular case, that is Lemma 4.5, Lemma~4.6 in \cite{ndiaye} and  references therein. To avoid technicalities, with a little abuse of notation we will write $o(1)$ to denote quantities which do not necessarily tend to zero but that can be made arbitrarily small. 
\begin{lemma} \label{lem-est}
Let $\varphi_{\lambda,\sigma}$ be as above. Then, for $\lambda\to+\infty$ it holds
\begin{align*}
	\int_M \varphi_{\lambda,\sigma}P^{2m}_g \varphi_{\lambda,\sigma}\,dvol_{g} & \leq  2k\Lambda_m(1+o(1)) \log\lambda,\\
	\int_M \Bigr(Q^{2m}_g+\dfrac{\Lambda_m}{2|M|} \sum_{j=1}^N \alpha_j  \Bigr) \varphi_{\lambda,\sigma}\,dvol_g & = -\kappa_{g_v}(1+o(1)) \log\lambda,\\
	\log\int_M \widetilde Qe^{2m\varphi_{\lambda,\sigma}}\,dvol_g & = O(1).
\end{align*}
\end{lemma}

\medskip

By the latter estimates we readily get the map we were looking for if we take $\kappa_{g_v}>k\Lambda_m$. Indeed, it is enough to observe that by Lemma \ref{lem-est} we have
$$
	\mathcal E(\Phi(\sigma)) \leq 4m(k\Lambda_m-\kappa_{g_v}+o(1))\log\lambda \to -\infty
$$	
as $\lambda\to+\infty$. Therefore, we can state the following result.
\begin{proposition} \label{prop-phi}
Suppose $\kappa_{g_v}>k\Lambda_m$ for some $k\in\mathbb{N}$. Then, for any $L>0$ there exists $\lambda\gg1$ such that $\Phi:\M_k\to\mathcal E^{-L}$.
\end{proposition}

\

We are now in position to prove the existence result.
\begin{proof}[Proof of Theorem \ref{th1.2}]
Suppose $\kappa_{g_v}\in(k\Lambda_m,(k+1)\Lambda_m)$ for some $k\in\mathbb{N}$ and $\kappa_{g_v}\notin\Gamma$, where $\Gamma$ is the critical set given in \eqref{gamma}. The proof is based on a min-max argument relying on the set $\M_k$ which will keep track of the topological properties of the low sublveles of the functional $\mathcal E$, jointly with the compactness property in Theorem~\ref{th1.1}. 

\

\noindent \emph{Step 1.} Recalling Lemma \ref{lem-psi}, let $L\gg1$ be such that there exists a continuous map $\Psi_R:\mathcal E^{-L}\to \M_k$. Then, by Proposition \ref{prop-phi} we can take $\lambda\gg1$ such that $\Phi:\M_k\to\mathcal E^{-L}$. Consider now the composition 
$$\begin{array}{ccccl}
\M_k&\stackrel{\Phi}{\longrightarrow}&\mathcal E^{-L}&\stackrel{\Psi_{R}}{\longrightarrow}& \M_k \hspace{2cm} \\
\sigma&\mapsto&\varphi_{\lambda,\sigma}&\mapsto&\Psi_{R}\left(\frac{\widetilde Qe^{2m\varphi_{\lambda,\sigma}}\,dvol_g}{\int_M \widetilde Q\,e^{2m\varphi_{\lambda,\sigma}}\,dvol_g}\right).
\end{array}$$
It is not difficult to see that the latter composition is homotopic to the identity map on $\M_k$. We have just to notice that, as $\lambda\to+\infty$, $\frac{\widetilde Qe^{2m\varphi_{\lambda,\sigma}}\,dvol_g}{\int_M \widetilde Q\,e^{2m\varphi_{\lambda,\sigma}}\,dvol_g}\rightharpoonup\sigma$ in the sense of measures, that $\Psi$ is a projection and that $R$ is a retraction onto $\M$. The homotopy is thus realized by letting  $\lambda\to+\infty$. As a consequence, if we consider the induced maps between homology groups $H_*$ we get that
\begin{equation} \label{inj}
H_*(\M_k) \hookrightarrow H_*\bigr(\mathcal E^{-L}\bigr) \mbox{ injectively}.
\end{equation}
Now, since $\M$ is a closed manifold, it is well-known that $\M_k$ has non-trivial homology groups and hence, in particular, it is non-contractible. We refer the interested readers for example to Lemma 3.7 in \cite{dm} where the $4$-dimensional case is considered. By the above discussion, this implies 
\begin{equation}\label{non-contr}
	\Phi(\M_k)\subset \mathcal E^{-L} \mbox{ is non-contractible}.
\end{equation}

\medskip

\noindent \emph{Step 2.} We next consider the topological cone over $\M_k$ which is defined through the equivalence relation
$$
	\mathcal C = \dfrac{\M_k\times[0,1]}{\M_k\times\{0\}},
$$
that is $\M_k\times\{0\}$ is collapsed to a single point which is the tip of the cone. We then define the min-max value
$$
	m=\inf_{f\in F} \sup_{\sigma\in\mathcal C} \mathcal E(f(\sigma)),
$$
where
$$
	F=\Bigr\{ f:\mathcal C\to H^m(M) \mbox{ continuous } \,:\, f(\sigma)=\varphi_{\lambda,\sigma}, \ \forall\sigma\in\partial\mathcal C=\M_k \Bigr\},
$$
which is non-empty since $t\Phi\in F$. Still by Proposition \ref{prop-phi} we can take $\lambda\gg1$ sufficiently large such that
$$
	\sup_{\sigma\in\partial\mathcal C} \mathcal E(f(\sigma))= \sup_{\sigma\in\M_k}\mathcal E(\varphi_{\lambda,\sigma})\leq -2L.
$$
On the other hand, we claim that 
$$
	m\geq-L.
$$
To prove it, we just need to observe that $\partial\mathcal C=\M_k$ is contractible in $\mathcal C$ (by construction of the cone) and thus $\Phi(\M_k)$ is contractible in $f(\mathcal C)$ for any $f\in F$. Hence, we deduce by \eqref{non-contr} that $f(\mathcal C)$ can not be contained in $\mathcal E^{-L}$, which proves the claim.

\medskip

We conclude that the functional $\mathcal E$ has a min-max geometry at the level $m$ which in turn implies there exists a Palais-Smale sequence at this level. 

\

\noindent \emph{Step 3.} Since the Palais-Smale condition is not available in this framework, we can not directly pass to the limit to obtain a critical point. To overcome this problem we use the so-called monotonicity trick jointly with the compactness property in Theorem \ref{th1.1}. This argument has been first introduced by Struwe in \cite{struwe} and has been then applied by many authors, see for example \cite{dm, ndiaye}. Therefore, we omit the details and just sketch the main ideas. 

\smallskip

One considers a small perturbation $\mathcal E_\e$ of the functional so that the above min-max scheme applies uniformly. By using a monotonicity property of the perturbed min-max values $m_\e$ it is possible to obtain a bounded Palais-Smale sequence which then converges to a solution of the perturbed problem. We then pass to the limit as $\e\to0$ by using the compactness property in Theorem \ref{th1.1} to eventually recover a solution of the original problem. This concludes the proof.
\end{proof}

\medskip

Finally, we present the proof of the multiplicity result in Theorem \ref{th-mult}.
\begin{proof}[Proof of Theorem \ref{th-mult}]
Once the above analysis (needed to prove the existence result) is carried out, the multiplicity result is essentially a straightforward application of Morse inequalities. Thus, we will be sketchy and refer for example to \cite{bdm,de} for further details. Recall also that $\mathcal E$ is assumed to be a Morse functional. The (weak) Morse inequalities would assert that
\begin{align*}
	\#\{ \mbox{solutions of }\eqref{1.mliouville}\} &\geq \sum_{q\geq0}\#\{ \mbox{critical points of $\mathcal E$ in $\{-L\leq\mathcal E\leq L\}$ with index $q$}\}  \\
	& \geq \sum_{q\geq0} \mbox{{dim} } 	H_{q}(\mathcal E^L, \mathcal E^{-L}),
\end{align*}
where $H_{q}(\mathcal E^L, \mathcal E^{-L})$ stands for the relative homology group of $(\mathcal E^L, \mathcal E^{-L})$, see for example Theorem 2.4 in \cite{de}. Now, it is known that the high sublevels $\mathcal E^L$ are contractible. Roughly speaking, one can take $L\gg1$ sufficiently large so that there are no critical points above the level $L$ which then allows to construct a deformation retract of $\mathcal E^L$ onto $H^m(M)$, which is of course contractible, see for example the argument in \cite{mal}. Then, by the long exact sequence of the relative homology, it easily follows that 
$$
H_{q}(\mathcal E^L, \mathcal E^{-L})\cong \widetilde H_q( \mathcal E^{-L}). 
$$
But we already now from \eqref{inj} that
$$
H_*(\M_k) \hookrightarrow H_*\bigr(\mathcal E^{-L}\bigr) \mbox{ injectively},
$$
thus 
$$
\mbox{{dim} } \widetilde H_q( \mathcal E^{-L})\geq \mbox{{dim} } \widetilde H_q( \M_k)
$$ 
and we are done.
\end{proof}

\

\begin{remark} \label{rem-general}
\emph{As already pointed out, for simplicity, all the argument has been carried out in the case $P_g^{2m}\geq0$. In general, one needs to modify the Adams-Trudinger-Moser inequality and its improvements by adding a bound to the component $u_{-}$ of the function $u$ lying in the direct sum of the negative eigenspaces of $P_g^{2m}$. As a consequence, in the low subleveles $\mathcal E^{-L}$ either the function $u$ concentrates or $u_{-}$ tends to infinity, or both alternative can hold. To express this alternative one can use the topological join (see  \cite{munkres1974topology}) between $\M_k$ and a set representing the negative eigenvalues of $P_g^{2m}$. We refer the interested reader to \cite{dm}.}
\end{remark}

\medskip

\section{Appendix: Pohozaev identity}
Here we state a Pohozaev-type identity which is used in the blow-up argument.

\begin{lemma}
\label{lea.poho}
Let $B_r(0)$ be a ball in $\Rm$, we have the following identities:
\begin{enumerate}	
\item [(a).] If $m=2m_0,~m_0\geq1$, then 
\begin{equation}
\label{a.poho1}
\begin{aligned}
&\int_{B_r(0)}x\cdot\nabla f(-\Delta)^{2m_0} fdx\\	
&=\sum_{i=2}^{m_0}\int_{\partial B_r(0)}
2(i-1)\left((-\Delta)^{m-i}f\frac{\partial(-\Delta)^{i-1}f}{\partial\nu}-\frac{\partial(-\Delta)^{m-i}f}{\partial\nu}(-\Delta)^{i-1}f\right)ds\\
&\quad+\sum_{i=1}^{m_0}\int_{\partial B_r(0)}\left.
(-\Delta)^{m-i}f\partial_{\nu}\langle x,\nabla(-\Delta)^{i-1}f\rangle ds\right.\\
&\quad-\sum_{i=1}^{m_0}\int_{\partial B_r(0)}\langle x,\nabla(-\Delta)^{i-1}f\rangle\frac{\partial(-\Delta)^{m-i}f}{\partial\nu}ds\\
&\quad+\int_{\partial B_r(0)}\frac12|x|((-\Delta)^{m_0}f)^2ds.
\end{aligned}
\end{equation}

\item [(b).] If $m=2m_0+1,~m_0\geq1$, then 	
\begin{equation}
\label{a.poho2}
\begin{aligned}
&\int_{B_r(0)}x\cdot\nabla f(-\Delta)^{2m_0+1} f\\	
&=\sum_{i=2}^{m_0}\int_{\partial B_r(0)}
2(i-1)\left((-\Delta)^{m-i}f\frac{\partial(-\Delta)^{i-1}f}{\partial\nu}-\frac{\partial(-\Delta)^{m-i}f}{\partial\nu}(-\Delta)^{i-1}f\right)ds\\
&\quad+\sum_{i=1}^{m_0}\int_{\partial B_r(0)}
(-\Delta)^{m-i}f\partial_{\nu}\langle x,\nabla(-\Delta)^{i-1}f\rangle ds\\
&\quad -\sum_{i=1}^{m_0}\int_{\partial B_r(0)}\langle x,\nabla(-\Delta)^{i-1}f\rangle\frac{\partial(-\Delta)^{m-i}f}{\partial\nu}ds\\
&\quad+\int_{\partial B_r(0)}\frac12|x|(\nabla(-\Delta)^{m_0}f)^2ds-2m_0\int_{\partial B_r(0)}\partial_{\nu}(-\Delta)^{m_0}f(-\Delta)^{m_0}fds\\
&\quad-\int_{\partial B_r(0)}\partial_{\nu}(-\Delta)^{m_0}f\langle x,\nabla(-\Delta)^{m_0}f\rangle ds.
\end{aligned}
\end{equation}	
\end{enumerate}
Here $\nu$ is outward normal vector along the boundary $\partial B_r(0).$
\end{lemma}
 
\begin{proof}
The proof is based on the following identity
$$(-\Delta)\langle x,\nabla(-\Delta)^if\rangle
=2(-\Delta)^{i+1}f+\langle x,\nabla(-\Delta)^{i+1}f\rangle.$$
Using the second Green's identity repeatedly, we can get above formula by straightforward computations. We omit the details.
\end{proof}

\

\

\bibliographystyle{acm}
\bibliography{biblio}

\end{document}